\documentclass[11pt]{article}
\usepackage{amsthm}
\usepackage{latexsym}
\usepackage[noblocks]{authblk}
\usepackage{amssymb,amsmath}  
\usepackage{enumitem}
\usepackage{multimedia}
\usepackage{subfigure}
\usepackage[noadjust]{cite}
\usepackage{authblk} 
\usepackage{mathrsfs} \usepackage{epsfig}
\usepackage{mathdots}
\usepackage{caption}
\usepackage{color}
\usepackage{pgfpages,tikz,tikz-cd,pgfkeys,pgfplots,calc}
\usetikzlibrary{shapes,snakes}
\usepackage{kotex}
\usepackage{titling}
\usepackage{setspace}
\usepackage{stmaryrd}
\usepackage[
linesnumbered,ruled,vlined]{algorithm2e}
\usepackage{algorithmicx, algpseudocode, algcompatible}
\definecolor{LemonChiffon}{rgb}{100, 98, 80}
\definecolor{myblue}{rgb}{0,0.4,0.8}
\definecolor{orange}{rgb}{1, 0.4, 0}
\definecolor{mygreen}{rgb}{0, 0.8, 0.2}
\definecolor{myred}{rgb}{204, 0, 0}
\definecolor{violet}{RGB}{0.4,0.2,1}
\definecolor{brown}{rgb}{0.6, 0.4, 0}

\newtheorem{theorem}{Theorem}[section]

\newtheorem{proposition}[theorem]{Proposition}
\newtheorem{corollary}[theorem]{Corollary}
\newtheorem{example}[theorem]{Example}

\theoremstyle{definition}
\newtheorem{definition}[theorem]{Definition}

\theoremstyle{definition}

\newcounter{statement}
\newcommand{\statement}[2]{%
 \begin{equation}\refstepcounter{statement}\tag{S\thestatement}\label{#1}%
  \parbox{\dimexpr\linewidth-4em}{#2}%
 \end{equation}
}

\begin{document}

\title{Linear-Time Computation of the Frobenius Normal Form for Symmetric Toeplitz Matrices via Graph-Theoretic Decomposition}
\author{Hojin Chu$^{1}$, Homoon Ryu$^{2}$ \\
 {\footnotesize $^{1}$ \textit{School of Computational Sciences,
Korea Institute for Advanced Study(KIAS),}}\\{\footnotesize\textit{
Seoul 02455, Rep. of Korea}}\\
{\footnotesize $^{2}$ \textit{Research Institute for Basic Science,
Ajou University,}}\\{\footnotesize\textit{
Suwon 16499, Rep. of Korea}}\\
{\footnotesize\textit{
hojinchu@kias.re.kr, ryuhomoon@ajou.ac.kr}}\\
{\footnotesize}}
\date{}
\maketitle

\begin{abstract}
We introduce a linear-time algorithm for computing the Frobenius normal form (FNF) of symmetric Toeplitz matrices by utilizing their inherent structural properties through a graph-theoretic approach.
Previous results of the authors established that the FNF of a symmetric Toeplitz matrix is explicitly represented as a direct sum of symmetric irreducible Toeplitz matrices, each corresponding to connected components in an associated weighted Toeplitz graph. 
Conventional matrix decomposition algorithms, such as Storjohann’s method (1998), typically have cubic-time complexity. 
Moreover, standard graph component identification algorithms, such as breadth-first or depth-first search, operate linearly with respect to vertices and edges, translating to quadratic-time complexity solely in terms of vertices for dense graphs like weighted Toeplitz graphs.
Our method uniquely leverages the structural regularities of weighted Toeplitz graphs, achieving linear-time complexity strictly with respect to vertices through two novel reductions: the \textit{$\alpha$-type reduction}, which eliminates isolated vertices, and the \textit{$\beta$-type reduction}, applying residue class contractions to achieve rapid structural simplifications while preserving component structure. 
These reductions facilitate an efficient recursive decomposition process that yields linear-time performance for both graph component identification and the resulting FNF computation. 
This work highlights how structured combinatorial representations can lead to significant computational gains in symbolic linear algebra.
\end{abstract}




\section{Introduction}

Toeplitz matrices arise frequently in numerical analysis, signal processing, and applied mathematics due to their structured form and favorable algebraic properties~\cite{gray2006toeplitz, T_bottcher2013analysis}. Their regularity enables efficient storage and operations, making them especially attractive for large-scale problems.
A central tool in understanding such matrices is the Frobenius normal form (FNF), which captures key algebraic invariants and is used in similarity classification and minimal polynomial analysis. However, traditional algorithms for computing the FNF, such as Storjohann’s cubic-time method~\cite{storjohann1998algorithms}, are computationally expensive and impractical for large structured matrices.

Recent work by Chu and Ryu~\cite{chu2024structural} established that the FNF of a symmetric Toeplitz matrix can be decomposed into a direct sum of symmetric irreducible Toeplitz matrices, each associated with a connected component in a certain weighted Toeplitz graph. This observation reveals a deep structural connection between algebraic matrix decomposition and combinatorial graph theory.

In this paper, we leverage this connection to develop a linear-time algorithm for computing the FNF of symmetric Toeplitz matrices. The core idea is to reinterpret the problem of matrix decomposition into a graph component identification problem in a highly regular, weighted graph. While general graph traversal algorithms such as breadth-first search or depth-first search typically require time proportional to the number of vertices and edges, this can be as large as quadratic in the number of vertices for dense graphs. In contrast, our method leverages the diagonal-constant and residue-class regularity of Toeplitz graphs to achieve linear-time complexity in the number of vertices.
Our approach is built on two graph reduction mechanisms:
\begin{itemize}
\item The \emph{$\alpha$-type reduction}, which eliminates isolated vertices to prevent unnecessary computation;
\item The \emph{$\beta$-type reduction}, which contracts vertices by residue classes, exploiting the periodic structure of Toeplitz diagonals while preserving connectivity.
\end{itemize}

These reductions form the foundation of a recursive decomposition strategy that simultaneously identifies graph components and realizes the FNF decomposition. The resulting algorithm thus unifies two objectives, graph analysis and matrix transformation, within a single efficient computational framework.
This unified perspective demonstrates how the algebraic and combinatorial features of Toeplitz matrices can be leveraged to improve both theoretical bounds and practical efficiency in canonical form computation.

The remainder of the paper is organized as follows. Section~2 introduces preliminaries and notation. Section~3 establishes the theoretical foundations of our reduction framework. Section~4 describes the algorithm and analyzes its complexity. Section~5 concludes with directions for future work.

\section{Preliminaries}

For graph-theoretical terms and notations not defined in this paper, we follow \cite{bondy2010graph}.

A {\it Toeplitz matrix} is a matrix in which each descending diagonal from left to right is constant.
That is, an $n\times n$ matrix $T=(t_{i,j})_{1\le i,j\le n}$ is a Toeplitz matrix if $t_{i,j}=t_{i+1,j+1}$ for each $1 \le i,j \le n-1$.
As a variation of the study of Toeplitz matrices in terms of graph theory, ``Toeplitz graph" was introduced and investigated with respect to hamiltonicity by van Dal {\it et al.}~\cite{G1_van1996hamiltonian} in 1966.
Results regarding Toeplitz graphs can be referenced in~\cite{G_cheon2023matrix,G_euler1995characterization,G_euler2013planar,G_heuberger2002hamiltonian,G_liu2019computing,G_mojallal2022structural,G_nicoloso2014chromatic}.
In \cite{chu2024structural}, the authors naturally generalized Toeplitz graphs to ``weighted Toeplitz graphs".

To define ``weighted Toeplitz graph", we explain some concepts.
The \emph{adjacency matrix} of a weighted graph $(G,w)$ is a square matrix $A$ of size $n \times n$ such that its element $a_{ij}$ is equal to the weight of an edge joining $v_i$ and $v_j$ when it exists, and zero when there is no edge joining $v_i$ and $v_j$. 
Obviously, the adjacency matrix of a graph is symmetric.
Note that given a symmetric matrix $A$, there is a weighted graph $(G, w)$ whose adjacency matrix is $A$, and $(G,w)$ is uniquely determined up to isomorphism.
Given two integers $m$ and $n$ with $m\le n$, let $\llbracket m,n \rrbracket$ be the set of integers between $m$ and $n$, that is,
\[ \llbracket m,n \rrbracket =\{k \in \mathbb{Z} \colon\, m \le k \le n\}. \]
We simply write $\llbracket 1,n \rrbracket$ as $[n]$.
In \cite{chu2024structural}, a weighted Toeplitz graph was defined as follows. 
\begin{definition}\label{def}
	Given an $n\times n$ symmetric Toeplitz matrix 
	\[A:=\begin{bmatrix}
  a_0 & a_{1}   & a_{2} & \cdots & \cdots & a_{n-1} \\
  a_1 & a_0      & a_{1} & \ddots &        & \vdots \\
  a_2 & a_1      & \ddots & \ddots & \ddots & \vdots \\ 
 \vdots & \ddots & \ddots & \ddots & a_{1} & a_{2} \\
 \vdots &        & \ddots & a_1    & a_0    & a_{1} \\
a_{n-1} & \cdots & \cdots & a_2    & a_1    & a_0
\end{bmatrix}, \]
we may denote $A$ by $T[a_0, a_1, \ldots, a_{n-1}]$ since a symmetric Toeplitz matrix is uniquely determined by its first row.  
We let \[S_A = \{ i \in \llbracket 0,n-1 \rrbracket \colon\, a_i \neq 0 \} \]
and $w_A$ be a map from $S_A$ to $\mathbb{R}\setminus \{0\}$ defined by 
\[
w_A(s) = a_s.
\]

Let $A= T[a_0, a_1, \ldots, a_{n-1}]$. 
The edge-weighted graph $(G,w)$ is uniquely determined by $A$ in the following way:
\[V(G) = [n], \quad E(G) = \{ ij \colon\, i,j \in V(G), \ |i-j| \in S_A\},\quad  \text{and}\quad  w(ij) = w_A(|i-j|). \]
This graph is denoted by $G(A)$. 
We mean by a {\it weighted Toeplitz graph} the edge-weighted graph of a symmetric Toeplitz matrix 
(see Figure~\ref{fig:wtex} for an illustration).
\end{definition}

Note that ``Toeplitz graph" introduced in~\cite{G1_van1996hamiltonian} is the weighted Toeplitz graph of a symmetric Boolean Toeplitz matrix.

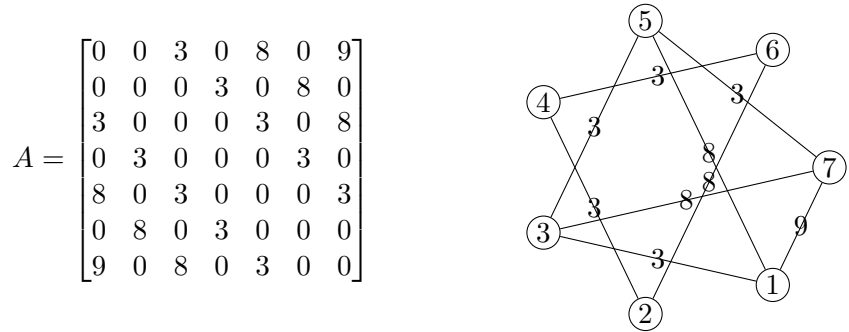
\begin{figure}
\begin{center}
\begin{equation*}
A=\begin{bmatrix}
0 & 0 & 3 & 0 & 8 & 0 & 9 \\
0 & 0 & 0 & 3 & 0 & 8 & 0 \\
3 & 0 & 0 & 0 & 3 & 0 & 8 \\
0 & 3 & 0 & 0 & 0 & 3 & 0 \\
8 & 0 & 3 & 0 & 0 & 0 & 3 \\
0 & 8 & 0 & 3 & 0 & 0 & 0 \\
9 & 0 & 8 & 0 & 3 & 0 & 0
\end{bmatrix}
\hspace{2cm}
\begin{tikzpicture}[baseline={(0,0)}, scale = 1]
\tikzstyle{vertex} = [circle,draw,fill=white,inner sep=1.3];
 \foreach \i in {1,...,7}{
      \node[vertex] (\i) at (-360*\i/7:2) {$\i$};
      }
\draw[black,>=stealth]  (1) -- (3) node[midway, ] {$3$};
\draw[black,>=stealth]  (2) -- (4) node[midway,] {$3$};
\draw[black,>=stealth]  (3) -- (5) node[midway, ] {$3$};
\draw[black,>=stealth]  (4) -- (6) node[midway, ] {$3$};
\draw[black,>=stealth]  (5) -- (7) node[midway, ] {$3$};
\draw[black,>=stealth]  (1) -- (5) node[midway, ] {$8$};
\draw[black,>=stealth]  (2) -- (6) node[midway, ] {$8$};
\draw[black,>=stealth]  (3) -- (7) node[midway, ] {$8$};
\draw[black,>=stealth]  (1) -- (7) node[midway, ] {$9$};

\end{tikzpicture}
\end{equation*}
\caption{A symmetric Toeplitz matrix $A= T[0,0,3,0,8,0,9]$ and its graph}\label{fig:wtex}
\end{center}
\end{figure}

Given a weighted graph $(G,w)$, a maximal connected subgraph of $(G,w)$ is called a \emph{component} of $(G,w)$.
A square matrix $A$ is called \emph{reducible} if by a simultaneous permutation of its lines, we can obtain a matrix of the form
\[
\begin{bmatrix}
A_{1} & A_{12} \\
O & A_{2}
\end{bmatrix}
\]
where $A_1$ and $A_2$ are square matrices of order at least one. 
If $A$ is not reducible, then $A$ is called \emph{irreducible}. 
The following well-known theorem gives a relation between the connectivity of a weighted graph and the irreducibility of its adjacency matrix. 

\begin{theorem}\label{thm:irreducible_connect}
A weighted graph is connected if and only if its adjacency matrix is irreducible.
\end{theorem}

In~\cite{chu2024structural}, the authors showed that a weighted Toeplitz graph is also component-wise weighted Toeplitz.
The vertex set of a component of a weighted Toeplitz graph may not have consecutive labels of the vertices. 
When the vertex set of a component is labeled as $n_1<n_2< \cdots< n_k$, relabeling it as $1, 2, \ldots, k$ is called {\it normalized labeling}.
Normalized labeling can be considered as a bijective map.

\begin{theorem}[\cite{chu2024structural}]\label{thm:sqz}
Each component of a weighted Toeplitz graph is isomorphic to a weighted Toeplitz graph under the normalized labeling of its vertex set.
\end{theorem}

Let $A$ be a square matrix of order $n$. 
Then there exists a permutation matrix $P$ of order $n$ and an integer $t \ge 1$ such that 
\begin{equation}\label{eq:PNF}
P A P^T = 
\begin{pmatrix}
A_1 & A_{12} & \cdots & A_{1t} \\
O & A_2 & \cdots &A_{2t} \\
\vdots & \vdots & \ddots & \vdots \\
O & O & \cdots & A_t
\end{pmatrix}
\end{equation}
where $A_1, A_2, \ldots, A_t$ are square irreducible matrices. 
The matrices in equation~\eqref{eq:PNF} occur as diagonal blocks and uniquely determined up to simultaneous permutation of their lines. 
The form in equation~\eqref{eq:PNF} is called the \emph{Frobenius normal form (FNF)} of the square matrix $A$. 
If $A$ is symmetric, then its Frobenius normal form is as follows:
\begin{equation}\label{eq:SYM}
P A P^T = 
\begin{pmatrix}
A_1 & O & \cdots &O \\
O & A_2 & \cdots &O \\
\vdots & \vdots & \ddots & \vdots \\
O & O & \cdots & A_t
\end{pmatrix}
\end{equation}
where $A_1, A_2, \ldots, A_t$ are symmetric irreducible matrices. 
Then, by Theorem~\ref{thm:irreducible_connect}, it is well-known that each of $A_1, A_2, \ldots, A_t$ in equation~\ref{eq:SYM} equals to the adjacency matrix of a component of the graph $(G,w)$ whose adjacency matrix is $A$.
In a matrix version, Theorem~\ref{thm:sqz} characterizes the FNF of a symmetric Toeplitz matrix.

\begin{theorem}[\cite{chu2024structural}]\label{thm:FNF}
For a symmetric Toeplitz matrix $A$, there exists a permutation matrix $P$ such that 
\[
P^TAP = 
\begin{bmatrix}
T_1 & O & \cdots & O \\
O & T_2 & \cdots & O \\
\vdots & \vdots & \ddots & \vdots \\
O & O & \cdots & T_k
\end{bmatrix}
\]
for some positive integer $k$ and irreducible symmetric Toeplitz matrices $T_1, \ldots, T_k$ such that $T_i$ is a principal submatrix of $T_j$ if and only if $i\ge j$.
\end{theorem}

\section{Graph Theoretical Framework}


We note that the problem of computing the FNF of a given symmetric Toeplitz matrix can be translated into the problem of identifying the components of corresponding weighted Toeplitz graph. 
In this section, we introduce several graph-theoretic results that facilitate efficient recursive reductions for graph component identification.

It is easy to see that the connectivity of vertices in a weighted graph is irrelevant to the weights or the presence of loops. 
Therefore it suffices to consider loopless weighted Toeplitz graphs with weight $1$ on each edge, that is, the graphs of symmetric Boolean Toeplitz matrices with zero main diagonal.
We note that given a positive integer $n$ and a subset $S$ of $[n-1 ]$, a symmetric Boolean Toeplitz matrix $A$ of order $n$ is uniquely determined so that $S_A = S$.

\begin{theorem}\label{thm:bigs}
Let $A = T[0, a_1, \ldots, a_{n-1}]$ be a symmetric Boolean Toeplitz matrix with zero main diagonal.
Suppose  $2s_0 > n$ where $s_0=\min S_A$.
Then $G(A) \cong G(T [a_m, \ldots, a_{n-1}]) \cup I_m$ which implies that $G(A)$ has $m$ more components than $G(T [a_m, \ldots, a_{n-1}])$ where $m = 2 s_0 - n$.
\end{theorem}

\begin{proof}
Take an edge $uv$ in $G(A)$ with $v<u$. 
Then $u = v + s$ for some $s \in S_A$.
Thus 
\begin{equation}\label{eq:uv}
u =v+s \ge 1+s \ge 1+s_0 \quad\text{and}\quad v=u-s \le n-s \le n-s_0. 
\end{equation}
Since $uv$ was arbitrarily chosen, every endpoint of an edge is contained in $\llbracket 1, n-s_0 \rrbracket \cup \llbracket 1+s_0,n \rrbracket$.
Thus $n-s_0+1, n-s_0+2, \ldots, s_0$ are isolated vertices in $G(A)$.

Now, we will show that $G(A)[W]$ is isomorphic to $G( T[a_m, \ldots, a_{n-1}])$ where \[W=\llbracket 1, n-s_0 \rrbracket \cup \llbracket 1+s_0,n \rrbracket.\]
We let $B= T[b_0, \ldots, b_{n-m-1}]$ where $b_i = a_{m+i}$ for each $i=0, 1, \ldots, n-m-1$. 
Then \[S_B = \{i \colon\, b_i \neq 0 \} = \{i \colon\, a_{m+i} \neq 0\} = \{i-m \colon\, a_i \neq 0\}\] and so 
\begin{equation}\label{eq:sb}
S_B = \{s-m \colon\, s \in S_A\}.
\end{equation}
Let $\varphi: W \rightarrow [n-m]$ be the normalized labeling of $G(A)[W]$. 
Then \[\varphi(w)=
\begin{cases}
	w & \text{if } w\le n-s_0 ;\\
	w-m &  \text{if } 1+s_0 \le w .
\end{cases}\]
To show that $\varphi$ is an isomorphism, take an edge $uv$ in $G(A)[W]$ with $v<u$.
Then since $G(A)[W]\subset G(A)$, $u-v \in S_A$.
Therefore $u \in \llbracket 1+s_0, n \rrbracket$ and $v \in \llbracket 1, n-s_0 \rrbracket $ as we have shown in \eqref{eq:uv}. 
Thus $\varphi(u) = u-m$ and $\varphi(v)=v$.
Now 
\[
\varphi(u)-\varphi(v) = (u-m)-v = (u-v)-m .    
\]
Since $u-v \in S_A$, $\varphi(u)-\varphi(v) \in S_B$ by \eqref{eq:sb}.
Therefore $\varphi(u)\varphi(v)$ is an edge in $G(B)$. 
Conversely, take an edge $ij$ in $G(B)$ with $i>j$.
Then $i-j \in S_B$, that is, there exists an element $s \in S_A$ satisfying $i-j = s-m$ by~\eqref{eq:sb}.
Thus \[i=j+(s-m) \ge j+(s_0-m) \ge 1+(s_0-m)=\varphi(1+s_0)\]
whereas
\[j = i-(s-m) \le (n-m)-(s-m) = n-s \le n-s_0=\varphi(n-s_0).\]
Since $\varphi$ is a normalized labeling, $\varphi^{-1}(i)=i+m$ and $\varphi^{-1}(j)=j$.
Then $\varphi^{-1}(i)-\varphi^{-1}(j) = (i+m)-j =s \in S_A$. 
Thus $\varphi^{-1}(i)\varphi^{-1}(j)$ is an edge in $G(A)[W]$.
Therefore $\varphi$ is an isomorphism from $G(A)[W]$ to $G(B)$.
\end{proof}

Given a graph $G$ with the vertex set $[n]$ and a positive integer $d$ with $d<n$, we say that $G$ is {\it $d$-reachable} if for any vertices $u$ and $v$ with $|u-v|=d$, there is a walk from $u$ to $v$. 
It is easy to check that if $G$ is $d$-reachable and $d \mid (u-v)$ for some vertices $u$ and $v$ in $G$, then there is an $uv$-walk in $G$. 
The following proposition is immediately true by the definition of Toeplitz matrix. 

\begin{proposition}\label{prop:sareach}
Given a symmetric Toeplitz matrix $A$, $G(A)$ is $s$-reachable for every $s\in S_A$.
\end{proposition}

To show the theorem that given a graph satisfying $s$-reachable and $t$-reachable is $\gcd(s,t)$-reachable (Theorem~\ref{prop:reach}), we see an observation from~\cite{G_cheon2023matrix}.
We have slightly modified the statement of Proposition~3.6 in~\cite{G_cheon2023matrix} for the usage in this paper.
While it is written in terms of digraph in~\cite{G_cheon2023matrix}, by replacing each arc with an edge in the digraph, the following proposition is immediately obtained.

\begin{proposition}[\cite{G_cheon2023matrix}]\label{prop:cycle}
Let $A = T[a_0, \ldots, a_{n-1}]$ with $a_s = a_{n-s}= 1$ and $a_i = 0$ if $i \notin \{s, n-s\}$ for a positive integer $s$ with $s \neq n-s$. 
Then $G(A)$ is a disjoint union of cycles each of which has the vertex set $\{i, i+d, \ldots, i+(n/d-1)d\}$ for some $i \in [d]$ where $d = \gcd(n,s)$. 
\end{proposition}

\begin{theorem}\label{prop:reach}
    Let $G$ be a graph with the vertex set $[n]$. 
    If $G$ is both $s$-reachable and $t$-reachable for some distinct positive integers $s$ and $t$ with $s+t \le n$, then $G$ is $\gcd(s,t)$-reachable. 
\end{theorem}

\begin{proof}
Suppose that $G$ is both $s$-reachable and $t$-reachable for some distinct positive integers $s$ and $t$ with $s+t \le n$. 
Let $A = T[a_0, \ldots, a_{s+t-1}]$ where $a_s = a_{t}= 1$ and $a_i = 0$ if $i \notin \{s, t\}$. 
Then $S_A = \{s,t\}$.
In addition, by Proposition~\ref{prop:cycle}, 
\statement{st:cycle}{
$G(A)$ is a disjoint union of cycles each of which has the vertex set \[\{i, i+d, \ldots, i+((s+t)/d-1)d\}\] for some $i \in [d]$ where $d = \gcd(s,t)$.}
Take vertices $u$ and $v$ satisfying $u-v = d$ in $G$. 
Since $n \ge s+t$ and $u = v + d \le v+(s+t)-1$, for $w: = \min\{v-1, n-(s+t)\}$, $u, v \in \llbracket w+1, w+(s+t) \rrbracket$. 
Therefore $u-w, v-w \in [s+t] = V(G(A))$. 
Then, since $(u-w)-(v-w) = u-v = d$, $u-w$ and $v-w$ are in the same component in $G(A)$ by \eqref{st:cycle}. 
Thus there is a walk $w_0w_1 \cdots w_\ell$ from $u-w$ to $v-w$ in $G(A)$ where $w_0 = u-w$ and $w_\ell = v-w$. 
Then, since $S_A = \{s, t\}$, $|w_{i+1}- w_i|  \in \{ s, t\}$ for each $i = 0, 1, \ldots, \ell-1$.
Moreover, $w+w_i \in [n]$ for each $i = 0, 1, \ldots, \ell$.
Then, since $G$ is $s$-reachable and $t$-reachable, there is a walk $W_i$  in $G$ from $w+w_i$ to $w+w_{i+1}$ for each $i = 0, 1, \ldots, \ell-1$. 
Thus there is a walk in $G$ from $u = w+w_0$ to $v = w+ w_\ell$ as follows:
\[
u W_0 (w+w_1) W_1 (w+w_2 )\cdots (w+w_{\ell-1}) W_{\ell-1} v.
\]
Since $u$ and $v$ were arbitrarily chosen, $G$ is $d$-reachable. 
\end{proof}

In the same vein as \cite{G_cheon2023matrix}, we define the contraction of a graph derived from residue classes. 

\begin{definition}
Let $G$ be a graph with the vertex set $[n]$ and $d$ be a positive integer. We define by $G/\mathbb{Z}_d$ the graph obtained from $G$ with the vertex set $\{[1]_d, [2]_d, \ldots, [d]_d\}$ where \[[i]_d = \{ v \in [n] \colon\, v \equiv i \pmod d\}\] and $[i]_d$ and $[j]_d$ are adjacent if and only if $xy$ is an edge of $G$ for some $x \in [i]_d$ and $y \in [j]_d$ (see Figure~\ref{fig:contr} for an illustration).
\end{definition}

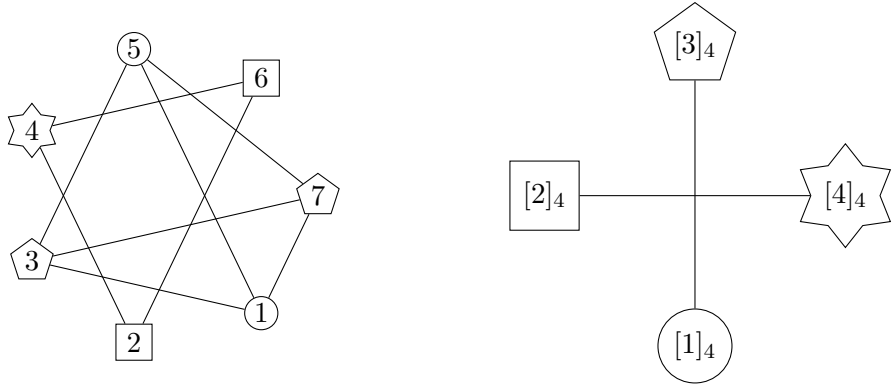
\begin{figure}
\begin{center}
\begin{tikzpicture}[baseline={(0,0)}, scale = 1]

\node[circle, draw, fill=white, inner sep=1.3] (1) at (-360/7:2) {$1$};
\node[circle, draw, fill=white, inner sep=1.3] (5) at (-360*5/7:2) {$5$};
\node[regular polygon, regular polygon sides = 4, draw, fill=white, inner sep=1.3] (2) at (-360*2/7:2) {$2$};
\node[regular polygon, regular polygon sides = 4, draw, fill=white, inner sep=1.3] (6) at (-360*6/7:2) {$6$};
\node[regular polygon, regular polygon sides = 5, draw, fill=white, inner sep=1.3] (3) at (-360*3/7:2) {$3$};
\node[regular polygon, regular polygon sides = 5, draw, fill=white, inner sep=1.3] (7) at (-360*7/7:2) {$7$};
\node[star, star points = 6, draw, fill=white, inner sep=1.3] (4) at (-360*4/7:2) {$4$};

\draw[black,>=stealth]  (1) -- (3);
\draw[black,>=stealth]  (2) -- (4);
\draw[black,>=stealth]  (3) -- (5);
\draw[black,>=stealth]  (4) -- (6);
\draw[black,>=stealth]  (5) -- (7);
\draw[black,>=stealth]  (1) -- (5);
\draw[black,>=stealth]  (2) -- (6);
\draw[black,>=stealth]  (3) -- (7);
\draw[black,>=stealth]  (1) -- (7);
\end{tikzpicture}
\hspace{2cm}
\begin{tikzpicture}[baseline={(0,0)}, scale = 1]

\node[circle, draw, fill=white, inner sep=3] (1) at (-360/4:2) {$[1]_4$};
\node[regular polygon, regular polygon sides = 4, draw, fill=white, inner sep=1.1] (2) at (-360*2/4:2) {$[2]_4$};
\node[regular polygon, regular polygon sides = 5, draw, fill=white, inner sep=1.1] (3) at (-360*3/4:2) {$[3]_4$};
\node[star, star points = 6, draw, fill=white, inner sep=1.1] (4) at (-360*4/4:2) {$[4]_4$};

\draw[black,>=stealth]  (1) -- (3);
\draw[black,>=stealth]  (2) -- (4);

\end{tikzpicture}

\caption{$G:=G(T[0,0,1,0,1,0,1])$ and its contraction $G/\mathbb{Z}_4$}\label{fig:contr}
\end{center}
\end{figure}

\begin{proposition}\label{prop:comp}
Let $G$ be a $d$-reachable graph. 
Then vertices $u$ and $v$ of $G$ are connected in $G$ if and only if vertices $[u]_d$ and $[v]_d$ are connected in $G/\mathbb{Z}_d$. 
\end{proposition}

\begin{proof}
By the definition of $G/\mathbb{Z}_d$, one may easy to check the ``only if" part.
To show the ``if'' part, take two vertices $u$ and $v$ in $G$ such that $[u]_d$ and $[v]_d$ are connected in $G/\mathbb{Z}_d$. 
Then there is a walk $[w_1]_d[w_2]_d\cdots [w_k]_d$ for some positive integer $k$ with $w_1 = u$ and $w_k = v$ in $G/\mathbb{Z}_d$. 
By the definition of $G/\mathbb{Z}_d$, for each $i \in [k-1]$, there is an edge $x_i y_i$ in $G$ with $x_i \equiv w_i \pmod d$ and $y_i \equiv w_{i+1} \pmod d$. 
We note that $x_{i+1} \equiv w_{i+1} \equiv y_i \pmod d$ for each $i = 0, \ldots, k-1$ where $y_0 = u$ and $x_{k} = v$. 
Since $G$ is $d$-reachable, there is a $y_ix_{i+1}$-walk $W_i$ in $G$ for each $i = 0, \ldots, k-1$. 
Therefore we have obtained the $uv$-walk 
\[
u = y_0W_0x_1y_1W_1x_2\cdots x_{k-1}y_{k-1}W_{k-1}x_{k} = v
\]in $G$. 
Thus $u$ and $v$ are connected in $G$.  
%
\end{proof}

The following corollary is immediately obtained from Proposition~\ref{prop:comp}.

\begin{corollary}\label{cor:rchcomp}
Let $G$ be a $d$-reachable graph. 
Then $G$ and $G/\mathbb{Z}_d$ have the same number of components. 
\end{corollary}

Let $A$ be a symmetric Boolean Toeplitz matrix of order $n$ with $2 \min S_A \le n$.
By Proposition~\ref{prop:sareach}, $A$ is $d_1$-reachable where $d_1 = \min S_A$. By Theorem~\ref{prop:reach}, the following is true:
\statement{st:di}{
if there is $s_i \in S_A$ with $s_i \le n-d_i$ and $d_i \nmid s_i$, then $A$ is $d_{i+1}$-reachable for each $i = 1, 2, \ldots$ where $d_{i+1}=\gcd(d_i,s_i)$. }
We note that $d_i$ in \eqref{st:di} strictly decreases as $i$ increases. 
Algorithm~1 is designed based upon this observation.
We note that 
\statement{st:da}{
the output $d_A$ is the smallest such that $G(A)$ is $d_A$-reachable to the extent that Theorem~\ref{prop:reach} can be used.
}

\begin{algorithm}[H]

    \caption{DA($n,S$) (Determine $d_A$)}
    \textbf{Input:} a positive integer $n$ and a nonempty subset $S$ of $[n-1]$ with $2\min S \le n$. 
    
    \textbf{Output:} \emph{${\rm DA}(n, S) = d_A$ described in \eqref{st:da} where $A$ is the symmetric Boolean Toeplitz matrix of order $n$ satisfying $S_A = S$.}
    
    \medskip
    
    $S_1 \leftarrow S - \{\min S\}$, $d_1 \leftarrow \min S$

    $i \leftarrow 1$
    
    \While{$S_i \neq \emptyset$ and $ \min S_i \le n-d_i$}{\label{lst:whlstrt}   
    
        $S_{i+1} \leftarrow S_i - \{\min S_i\}$\label{lst:o1time}
        
        $d_{i+1} \leftarrow \gcd(d_i, \min S_i)$ 

        $i \leftarrow i+1$
    
    }\label{lst:whlend}

    $d_A \leftarrow d_i$
    
\end{algorithm}

Using Algorithm~1, we could obtain an useful theorem as follows. 

\begin{theorem}\label{thm:main}
Let $A=T[0, a_1, \ldots, a_{n-1}]$ be a symmetric Boolean Toeplitz matrix with $2 \min S_A \le n$. 
Let $d={\rm DA}(n,S_A)$. 
In addition, let $q$ and $r$ be integers satisfying $n=qd+r$ and $0 \le r <d$. 
Then \[G(A)/\mathbb{Z}_{d} \cong G(B)/\mathbb{Z}_{d}\] where $B = T[0, b_1, \ldots, b_{d+r-1}]$ with
\[
b_i = 
\begin{cases}
0 & \text{if } i \le r; \\
a_{(q-1)d+i} & \text{if } r < i < d+r \text{ and } i \neq d; \\
1 & \text{if } i = d.
\end{cases}
\]
Further, $G(A)$ and $G(B)$ have the same number of components.
\end{theorem}

\begin{proof}
We consider a mapping $\varphi$ from $V(G(A)/\mathbb{Z}_d)$ to $V(G(B)/\mathbb{Z}_d)$ sending $[i]_d$ to $[i]_d$. 
It is easy to check that $\varphi$ is a bijection. 

To show $E(G(B)/\mathbb{Z}_d) \subseteq \varphi(E(G(A)/\mathbb{Z}_d))$, take an edge $[i]_d[j]_d$ in $G(B) /\mathbb{Z}_d$ with $i,j \in [d]$. 
Then $u-v= :s \in S_B$ for some vertices $u$ and $v$ in $G(B)$ with $u \equiv i \pmod d$ and $v \equiv j \pmod d$. 
Since $i$ and $j$ are distinct, $d \nmid s$ and so $s \neq d$.
Thus, by the definition of $B$, $a_{(q-1)d+s}=b_s=1$.
Since $u \le d+r$, $u+(q-1)d \le d+r+(q-1)d =n$.
Then $u+(q-1)d$ and $v$ are vertices of $G(A)$.
Since $(u+(q-1)d)-v=(q-1)d+(u-v)=(q-1)d+s$ and $a_{(q-1)d+s}=1$, $u+(q-1)d$ and $v$ are adjacent in $G(A)$.
Therefore $[i]_d[j]_d$ is an edge in $G(A) /\mathbb{Z}_d$ since $u+(q-1)d \equiv u \equiv i \pmod d$ and $v \equiv j \pmod d$.
Thus we have shown $E(G(B)/\mathbb{Z}_d) \subseteq \varphi(E(G(A)/\mathbb{Z}_d))$.

To show $\varphi(E(G(A)/\mathbb{Z}_d)) \subseteq E(G(B)/\mathbb{Z}_d)$, take an edge $[i]_d[j]_d$ in $G(A) /\mathbb{Z}_d$ with $i,j \in [d]$. 
Then $u-v= :s \in S_A$ for some vertices $u$ and $v$ in $G(A)$ with $u \equiv i \pmod d$ and $v \equiv j \pmod d$. 
We note that $a_s = 1$. 
Since $i$ and $j$ are distinct, $d \nmid s$.
Since $d$ terminates the while loop of Algorithm~1, $s>n-d$. 
Then \[r = n-qd < s-(q-1)d < n-(q-1)d = d+r.\] 
Let $t = s-(q-1)d$. 
Then $b_t=a_{(q-1)d+t} =a_s = 1$ by the definition of $B$. 
We note that
\[(q-1)d \le n-d < s \le v+s = u \le n = qd+r. \]
Thus $(q-1)d < u \le qd+r$ and so $ 0< u-(q-1)d \le d+r$.
On the other hand, we have $1\le v=u-s \le n-s <n-(n-d) = d$.  
Thus $1 \le v \le d$.
Therefore \[\{u-(q-1)d ,\ v\} \subseteq [d+r].\] Since $(u-(q-1)d)-v = s-(q-1)d = t$ and $b_{t}=1$, $u-(q-1)d$ and $v$ are adjacent in $G(B)$. 
Since $i \equiv u \equiv u-(q-1)d \pmod d$ and $j \equiv v \pmod d$, $[i]_d[j]_d$ is an edge in $G(B)/\mathbb{Z}_d$.
Thus we have shown $\varphi(E(G(A)/\mathbb{Z}_d)) \subseteq E(G(B)/\mathbb{Z}_d)$.
Consequently, $\varphi$ is an isomorphism.
Hence \[G(A)/\mathbb{Z}_d \cong G(B)/\mathbb{Z}_d. \]

By \eqref{st:da}, $G(A)$ is $d$-reachable. 
Thus, by Corollary~\ref{cor:rchcomp}, $G(A)$ and $G(B)$ have the same number of components.
\end{proof}

\begin{example}\label{ex:block}

Let $A = T[0, a_1, \ldots, a_{30}]$ be a symmetric Boolean Toeplitz matrix with \[
a_i=1 \quad\text{if }i \in \{12, 18, 24, 29\}.
\]
Now, we determine the number of blocks in the FNF of $A$ by applying Theorems~\ref{thm:bigs} and~\ref{thm:main}.

By applying Theorem~\ref{thm:main} to $A$, we have $G(A)/\mathbb{Z}_6 \cong G(B)/\mathbb{Z}_6$ where $B = T[0, b_1, \ldots, b_6]$ with $b_1=b_2=b_3 = b_4=0$ and $b_5=b_6=1$. 
Further, by applying Theorem~\ref{thm:bigs} to $B$,  $G(B)\cong G(C) \cup I_3$ where $C = T[0, 0, 1, 1]$. 
It is easy to check that $G(C)$ is connected. 

We note that $\{1, 2, 3, 4\}$ is the only component in $G(C)$. 
Since $G(B)\cong G(C) \cup I_3$, $\{3, 4, 5\}$ are isolated vertices in $G(B)$ and $\{1, 2, 6, 7\}$ is a component in $G(B)$. 
Since $G(A)$ is $6$-reachable and $G(A)/\mathbb{Z}_6 \cong G(B)/\mathbb{Z}_6$,  the components of $G(A)$ are 
\[\{3, 9, 15, 21, 27\}, \quad \{4, 10, 16, 22, 28\}, \quad \{5, 11, 17, 23, 29\},
\]
and
\[
\{1, 2, 6, 7, 8, 12, 13, 14, 18, 19, 20, 24, 25, 26, 30, 31\}. 
\]
Moreover, the FNF of $A$ is 
\[
\begin{bmatrix}
X & O & O & O\\
O & X & O & O\\
O & O & X & O \\
O & O & O & Y
\end{bmatrix}
\]
where $X = T[0, 0, 1, 1 ,1]$ and $Y = T[0,y_1, \ldots, y_{18}]$ with $y_i = 1$ if and only if $i \in \{ 6, 9, 12, 17\}$.
\end{example}

\section{Proposed Algorithm}

%

In this chapter, we present a linear-time algorithm which determines the FNF of a symmetric Toeplitz matrix.
Moreover, we explain how the algorithm works and how long it takes. 

Let $A$ be a symmetric Toeplitz matrix of order $n$.
Note that a symmetric Toeplitz matrix is uniquely determined by its first row.
In this vein, we could consider a symmetric Toeplitz matrix as its first row and vice versa.
More specifically, given $A$, we use the first row of $A$ as input data. 
Moreover, to identify the FNF of $A$, it suffices to determine the first row of each diagonal block by Theorem~\ref{thm:FNF}.

In Subsection~\ref{sec:reduction}, by applying Theorems~\ref{thm:bigs} and~\ref{thm:main}, we reduce $A$ to a smaller matrix, while simultaneously extracting information derived from the difference in the number of components of their graphs. 
By repeatedly performing the above reduction, we eventually reach a zero matrix.
In Subsection~\ref{sec:identifying}, we trace the reduction process in reverse to identify the vertices in a component of $G(A)$. 
Finally, we determine the first row of each diagonal block in the FNF of $A$ in Subsection~\ref{sec:determining}. 
These three steps together form a complete algorithm for computing the FNF.
Each step is presented as an individual algorithm, designed to be applied sequentially.
In Subsection~\ref{sec:complexity}, we analyze the time complexity of each step and show that all three algorithms run in linear time, yielding an overall linear-time procedure.

The algorithms designed in this section work correctly regardless of whether $A$ is a boolean matrix with zero diagonal or not.
For simplicity, we will explain the case where $A$ is a boolean matrix with zero diagonal.

\subsection{Graph reduction algorithm}\label{sec:reduction}

For notational convenience, we call a reduction of a Toeplitz matrix an {\it $\alpha$-type reduction} (resp. a {\it $\beta$-type reduction}) if it is an application of Theorem~\ref{thm:bigs} (resp. Theorem~\ref{thm:main}).
Based on these reductions, we design Algorithm~2 as follows.


\begin{algorithm}[H]
\caption{Reduction of $A$}
\textbf{Input:} $A$

\textbf{Intermediate data:} \emph{$(n_j)_{j=1}^t$, $(c_j)_{j=1}^{t-1}$, and $(d_j)_{j=1}^{t-1}$}

\textbf{Output:} \emph{ $c$}

\medskip

read $S_A$, in increasing order, from $A$.\label{lst:init}

$n_1 \leftarrow n$, $S_1 \leftarrow S_A-\{0\}$

$i \leftarrow 1$

\While{$S_i \neq \emptyset$}{\label{lst:whostrt}
\vspace{0.2cm}

    {\bf($\alpha$-type reduction)}
    
    \If{ $2 \min S_i > n_i$  } 
    {\label{lst:ifstrt}
        
        $d_i \leftarrow \min S_i$

        $m \leftarrow 2d_i - n_i$
        
        $n_{i+1} \leftarrow n_i-m$, $S_{i+1} \leftarrow \{s-m \colon s \in S_i\}$, $c_{i} \leftarrow m$
    }\label{lst:ifend}
\vspace{0.3cm}

    {\bf($\beta$-type reduction)}
    
    \Else{\label{lst:elsestrt}
    
    $d_i \leftarrow {\rm DA}(n_i,S_i)$, $q \leftarrow \lfloor n_i / d_i\rfloor$
        
    \If{ $d_i\nmid n_i$}{
         $r \leftarrow n_i - d_iq$
        
        $n_{i+1} \leftarrow d_i+r$, $S_{i+1} \leftarrow \{s-(q-1)d_i \colon s \in S_i, s > n_i-d_i \} \cup \{d_i\}$  
    }\label{lst:elseend2}
    
    \Else{
        $n_{i+1} \leftarrow d_i$, $S_{i+1} \leftarrow \{s-(q-1)d_i \colon s \in S_i, s > n_i-d_i \}$ 
    }
    
        $c_i \leftarrow 0$
    
    }\label{lst:elseend}

    $i \leftarrow i+1$
}\label{lst:whoend}

$t \leftarrow i$

$c \leftarrow \sum_{j=1}^{t-1} c_j + n_t$

\end{algorithm}

In Line~\ref{lst:init}, we read the first row of $A$ to list the elements in $S_A$ in increasing order. 
An $\alpha$-type reduction deals with the case $n_i< 2 \min S_i $ and a $\beta$-type reduction deals with the case $n_i \ge 2 \min S_i $, so the while loop is designed to be divided into the cases $n_i < 2\min S_i$ (Lines~\ref{lst:ifstrt}--\ref{lst:ifend}) and $n_i \ge 2 \min S_i$ (Lines~\ref{lst:elsestrt}--\ref{lst:elseend}). 
To describe how Theorems~\ref{thm:bigs} and ~\ref{thm:main} work, 
\statement{stai}{let $A_i$ be the symmetric Boolean Toeplitz matrix of order $n_i$ so that $S_{A_i} = S_i$. }

In an $\alpha$-type reduction (Lines~\ref{lst:ifstrt}--\ref{lst:ifend}), by Theorem~\ref{thm:bigs}, $G(A_i)$ has at least $m:= 2\min S_i - n_i$ isolated vertices, which are trivial components. 
Moreover, $G(A_{i}) \cong G(T[a_m, \ldots, a_{n_i-1}]) \cup I_m$.
In Line~\ref{lst:ifend}, we take $n_{i+1}$ and $S_{i+1}$ so that $A_{i+1}= T[a_m, \ldots, a_{n_i-1}]$ and we save the loss of $m$ trivial components to $c_i$. 
In a $\beta$-type reduction (Lines~\ref{lst:elsestrt}--\ref{lst:elseend}), the variables $n_{i+1}$ and $S_{i+1}$ are taken so that Theorem~\ref{thm:main} implies that $G(A_i)$ and $G(A_{i+1})$ have the same number of components. 

In Lines~\ref{lst:ifend}, \ref{lst:elseend2}, and \ref{lst:elseend}, we naturally deduce $S_{i+1}$ from $S_i$ in increasing order for each $i$. 
Moreover, one may check that the sum of the elements in $S_{i+1}$ is strictly less than that in $S_i$.
Thus the while loop terminates in at most $\sum_{s \in S_A} s$ iterations.

When the loop is terminated, $A_{i}=T[0,0,\ldots, 0]\cong I_{n_i}$, that is, $A_i$ consists of exactly $n_i$ trivial components.
Therefore we may conclude that the number of components of $G(A)$ equals 
\[c=\sum_{j=1}^{t-1} c_j + n_t. \] 

We introduce and clarify several intermediate quantities that emerge during its execution. These terms encapsulate essential structural information and will be instrumental in both theoretical developments and practical implementations discussed later.
First, we denote by $n_j$ the size of the matrix produced at the $j$-th step of the reduction algorithm. 
As previously discussed, $c_j$ denotes the decrease in the number of components at step $j$. 
Moreover, $d_j$ denotes the minimum element of $S_j$ and we note that $A_j$ is $d_j$-reachable.

\subsection{Identifying the vertices in a component}\label{sec:identifying}

The {\it component index sequence (or shortly, CIS)} is a map $\rho$ from $[n]$ to $[c]$ such that $\rho(u) = \rho(v)$ if and only if $u$ and $v$ are in the same component. 
We design an algorithm for a component index sequence of $A$ from intermediate data $(n_j)_{j=1}^i$, $(c_j)_{j=1}^{i-1}$, and $(d_j)_{j=1}^{i-1}$.

\begin{algorithm}[H]
\caption{Find a component index sequence of $A$}
\textbf{Input:} Intermediate data $(n_j)_{j=1}^t$, $(c_j)_{j=1}^{t-1}$, and $(d_j)_{j=1}^{t-1}$, from the reduction algorithm.
    
\textbf{Output:} \emph{the component index sequence $\rho : [n] \rightarrow [c]$}
    
\medskip
    
$\rho_t (j) = j$ for $j=1, 2, \ldots, n_t$. \label{lst:rhoi}

$M \leftarrow n_t$

\For{$j \leftarrow t-1$ to $1$}{
\vspace{0.2cm}

    {\bf(Recover $\alpha$-type reduction)}
    
\If{$c_j\neq 0$}{

$\delta \leftarrow n_j - n_{j+1}$

\For{$p \leftarrow 1$ to ${n_{j+1}}$}
{

\If{$p \le \frac{n_{j+1}}{2}$}{ \label{lst:alp1st}
$\rho_j(p) \leftarrow \rho_{j+1}(p)$ \label{lst:alp1ed}
}

\Else{\label{lst:alp2st}
$\rho_{j}(p+\delta)\leftarrow \rho_{j+1}(p)$ \label{lst:alp2ed}
}

}

\For{$p \leftarrow 1$ to $\delta$}{\label{lst:alp3st}
$\rho_{j}( \frac{n_{j+1}}{2} + p) \leftarrow M+p$
\label{lst:alp3ed}
}

$M \leftarrow M+(n_{j}-n_{j+1})$

}

\vspace{0.2cm}

    {\bf(Recover $\beta$-type reduction)}

\Else{ \label{lst:betst}

\For{$p\leftarrow 1$ to $n_{j+1}$}{
    $\rho_{j}(p) \leftarrow \rho_{j+1}(p)$
}

\For{$p \leftarrow n_{j+1}+1$ to $n_{j}$}{

    $\rho_{j}(p) \leftarrow \rho_{j}(p-d_{j})$
\label{lst:beted}
}

}

$\rho \leftarrow \rho_1$
}
    
\end{algorithm}

In Algorithm~3, we follow up the process of Algorithm~2 in reverse direction from $A_t$ to $A_1$ (see \eqref{stai}). 
Thus the for loop proceeds with the index $j$ iterating backward from $t-1$ to $1$. 
We note that $G(A_t)$ is the isolated graph of $n_t$ vertices. 
Therefore the component index sequence of $A_t$ is equal to $\rho_t$ given in Line~\ref{lst:rhoi}.

After an $\alpha$-type reduction, the number of components decreases, but after a $\beta$-type reduction does not changes the number of components. 
Since $c_j$  is equal to the difference in the number of components between $G(A_j)$ and $G(A_{j+1})$, the $j$-th reduction is $\alpha$-type or $\beta$-type based on whether $c_j$ equals $0$ or not, respectively. 

An $\alpha$-type reduction eliminates isolated vertices whose labels form a consecutive subsequence.
In this process, the number of vertices with labels greater than the removed ones is equal to the number of vertices with labels smaller than the removed ones. 
Therefore, to restore the CIS after performing an $\alpha$-type reduction, one can take the $(j+1)$-th CIS, split it in half, place the front half at the beginning (Lines~\ref{lst:alp1st}--\ref{lst:alp1ed}) and the back half at the end (Lines~\ref{lst:alp2st}--\ref{lst:alp2ed}), and then fill the center with elements not present in the $(j+1)$-th CIS, ensuring they are isolated (Lines~\ref{lst:alp3st}--\ref{lst:alp3ed}).

A $\beta$-type reduction replaces a $d$-reachable graph with fewer vertices than the original, but still at least $d$ vertices. 
Therefore, to restore the CIS after a $\beta$-type reduction, one can repeatedly copy the first $d$ values of the reduced CIS and continue appending them iteratively  (Lines~\ref{lst:betst}--\ref{lst:beted}).

\subsection{Computing the Frobenius normal form}\label{sec:determining}

Recall that a symmetric Toeplitz matrix is uniquely determined by its first row. Based on this observation, we introduce an algorithm that identifies the first row of each symmetric Toeplitz diagonal block $T_i$, thereby fully determining the entire FNF of $A$ (see Theorem~\ref{thm:FNF}).

\begin{algorithm}[H]
\caption{Determining each blocks of the FNF of given symmetric Toeplitz matrix}
\textbf{Input:} the CIS $\rho$ and the first row $A_{1,*}$ of given symmetric Toeplitz matrix $A$.
    
\textbf{Output:} \emph{$B$ }

\medskip
    
$B$ be a $c \times n$ matrix with null entries. 

$m, \ell$ be sequences of length $c$ with zero entries.

\For{$v\leftarrow 1$ to $n$}{

\If{$m(\rho(v)) = 0$}{

$m(\rho(v)) \leftarrow v$

}

$\ell(\rho(v)) \leftarrow \ell(\rho(v))+1$

$B_{\rho(v), \ell(\rho(v))} \leftarrow A_{1,v-m(\rho(v))+1}$ \label{lst:normal}

}

\end{algorithm}

As the output of the Algorithm~4, the matrix $B$ is constructed such that the $i$-th row corresponds to the first row of the diagonal block $T_i$ in FNF of $A$. 
Here, $m$ and $\ell$ serve as auxiliary sequences used to construct the matrix $B$. 
The value $m(i)$ corresponds to the label of the first (i.e., smallest-labeled) vertex assigned to block $i$, while $\ell(i)$ records the number of vertices already assigned to block $i$, thereby indicating the current insertion position within the block. 
The value $\ell(i)$ can also be viewed as a normalized labeling in the $i$-th component. 
Therefore, by Theorem~\ref{thm:sqz}, Line~\ref{lst:normal} effectively implements the isomorphism induced by normalized labeling.
Hence Algorithm~4 computes the first row of each block of FNF of $A$.

\subsection{Efficiency}\label{sec:complexity}

{\bf The running time of Algorithm~2.}
To compute the running time of Algorithm~2, we first analyze the running time of Algorithm~1. 
In each iteration of the while loop from Line~\ref{lst:whlstrt} to Line~\ref{lst:whlend} in Algorithm~1, computing $\gcd(d_i,\min S_i)$ can be implemented in time $\mathcal{O}(\log(d_i/d_{i+1}))$. 
Thus, each iteration requires $\mathcal{O}(1) + \mathcal{O}(\log(d_i/d_{i+1}))$. 
Since the number of iterations of the while loop is at most  $\mathcal{O}(|S_i|)$, the running time of Algorithm~1 is at most \[\mathcal{O}(|S_i|) + \mathcal{O}(\log(d_1/d_A)) = \mathcal{O}(n_i)\]
(where $|S_i| \le n_i$ and $d_1/d_A \le n_i$). 

Next, we compute the time complexity of Algorithm~2. 
If $n_i< 2\min S_i$, then $n_{i+1} \ge 2\min S_{i+1}$, and thus an $\alpha$-type reduction is followed by a $\beta$-type reduction.
Therefore a $\beta$-type reduction occurs at least once in every two iteration of the while loop. 
In a $\beta$-type reduction, since $d_i \le \min S_i \le n_i/2$, we have  
\[
\frac{n_{i+1}}{n_i} = \frac{d_i+r}{qd_i+r} \le \frac{d_i+r}{2d_i+r} = 1-\frac{d_i}{2d_i+r} <\frac{2}{3}
\]
and thus $n_i$ decreases geometrically (or faster). 
Each iteration of the while loop, whether $\alpha$-type or $\beta$-type, can be implemented in time $\mathcal{O}(n_i)$ time, since $|S_i| \le n_i$.  
Therefore, we conclude that the running time of Algorithm~2 is $\mathcal{O}(n)$.

{\bf The running time of Algorithm~3.}
Each iteration of the outer loop (indexed by $j$) incurs a constant cost of $\mathcal{O}(n_j)$, where $n_j$ represents the size of the corresponding sequence. 
Therefore the overall time complexity of the algorithm can be expressed as:
\[
\mathcal{O}\left(\sum_{j=1}^{t-1}n_j\right).
\]
Given that $n_j$ decreases geometrically (or a faster rate), the summation is bounded by a constant multiple of $n$.
Hence, the overall time complexity is $\mathcal{O}(n)$. 

It is straightforward to see that Algorithm~4 runs in $\mathcal{O}(n)$ time.
Since each of Algorithms~2, 3, and 4 runs in linear time. 
Their sequential application yields the FNF of $A$, and the entire procedure thus takes $\mathcal{O}(n)$ time.

\section{Concluding Remark}
Our results are based on the structural properties inherent to Toeplitz matrices. 
A closely related matrix structure is the Hankel matrix, which shares many similarities with a Toeplitz matrix. 
In fact, Chu and Ryu~\cite{chu2024structural} demonstrated that not only the FNF (Frobenius Normal Form) of a symmetric Toeplitz matrix is the direct sum of symmetric irreducible Toeplitz matrices, but also the FNF of a Hankel matrix is the direct sum of irreducible Hankel matrices. 
Based on this observation, we expect that there also exists an efficient algorithm for computing the FNF of Hankel matrices.

\section{Acknowledgement}
Hojin Chu was supported by a KIAS individual Grant (CG101801) at Korea Institute for advanced Study.
Homoon Ryu was supported by the National Research Foundation of Korea (NRF) grant funded by the Korea government (MSIT) (No. RS-2025-00523206).


\end{document}